\documentclass[preprint,12pt,3p]{elsarticle}

\usepackage{amssymb}
\usepackage{amsmath,amsthm}
\usepackage{mathrsfs}
\usepackage{hyperref}
\usepackage{cleveref}
\usepackage[all,cmtip]{xy}
\usepackage{epsf, graphicx}
\usepackage{latexsym,amsfonts,amsbsy,amssymb}
\usepackage{geometry}
\usepackage{titletoc}
\usepackage{color,xcolor}
\usepackage{rotating}
\usepackage{algorithm}
\usepackage{algorithmic}
\usepackage{amscd}

\makeatletter
\renewcommand\theequation{\thesection.\arabic{equation}}
\@addtoreset{equation}{section} \makeatother
\newtheorem{theorem}{Theorem}[section]
\newtheorem{lemma}[theorem]{Lemma}
\newtheorem{notation}[theorem]{Notation}

\newtheorem{remark}[theorem]{Remark}
\newtheorem{void}[theorem]{}
\newtheorem{example}[theorem]{Example}
\newtheorem{definition}[theorem]{Definition}
\newtheorem{proposition}[theorem]{Proposition}
\newtheorem{statement}[theorem]{Statement}

\def\IBr{{\rm IBr}}
\def\Irr{{\rm Irr}}
\def\Ind{{\rm Ind}}
\def\Res{{\rm Res}}

\def\Aut{{\rm Aut}}

\def\End{{\rm End}}

\def\Id{{\rm Id}}
\def\G{\mathcal{G}}

\def\O{\mathcal{O}}

\def\F{\mathbb{F}}
\def\FF{\mathcal{F}}

\def\Z{\mathbb{Z}}

\def\Tr{{\rm Tr}}

\def\SU{{\rm SU}}
\def\SL{{\rm SL}}
\def\GL{{\rm GL}}

\makeatletter
\def\ps@pprintTitle{%
\let\@oddhead\@empty
\let\@evenhead\@empty
\def\@oddfoot{\reset@font\hfil\thepage\hfil}
\let\@evenfoot\@oddfoot
}
\makeatother

\begin{document}

\begin{frontmatter}

\title{Morita equivalent blocks with an endotrivial source that are not globally isotypic}

\author{Xin Huang}


\begin{abstract}
Let $(K,\mathcal{O},k)$ be a large enough $2$-modular system. We construct two nilpotent block algebras of finite groups which are isomorphic as $\mathcal{O}$-algebras and are Morita equivalent via a bimodule with an endotrivial source of $\mathcal{O}$-rank $3$, but which are not globally isotypic. Hence they are not $2$-permutation equivalent and not splendidly Rickard equivalent. These results separate Morita equivalences with endopermutation source from $p$-permutation equivalences and splendid Rickard equivalences, and also separate almost global isotypies from global isotypies.

\end{abstract}


\end{frontmatter}


\section{Introduction}\label{s1}

Throughout this paper $p$ is a prime, $k$ is an algebraically closed field of characteristic $p$, and $\O$ is a complete discrete valuation ring with quotient field $K$ of characteristic $0$ and residue field $k$. Assume that $K$ is a splitting field for all finite groups considered below.

The notion of an isotypy between two blocks of finite groups was first defined by Brou\'{e} in \cite[D\'efinition 4.6]{Broue1990} when he proposed his celebrated abelian defect group conjecture. With the development of Brou\'{e}'s conjecture, Rickard \cite{Rickard} defined the notion of a splendid equivalence, and later Boltje and Xu \cite{BX} defined the notion of a $p$-permutation equivalence. Brou\'{e}'s original definition of an isotypy asks for the compatibility of generalized decomposition maps with respect to centralizers of cyclic $p$-subgroups. In order to distinguish different notions, let us  call it a {\it cyclic isotypy}.  When studying $p$-permutation equivalences, stronger versions of isotypies were defined; see \cite[1.9 (b)]{BX}, \cite[§10]{Lin08} and \cite[Definition 15.3]{BP}. The notions of an isotypy in \cite[§10]{Lin08} and \cite[Definition 15.3]{BP} are equivalent, although they are formulated in terms of slightly different choices of local data. In this paper, let us call an isotypy in \cite[Definition 15.3]{BP} a {\it global isotypy}; this will be reviewed in Definition \ref{defi: almost global isotypy} below. In \cite{BX,Lin08,BP}, it is shown that splendid Rickard equivalence $\Rightarrow$ $p$-permutation equivalences $\Rightarrow$ global isotypies $\Rightarrow$ cyclic isotypies.

Morita equivalences between blocks play a central role in another two major conjectures in representation theory of finite groups, namely Donovan's conjecture (\cite[Conjecture 6.1.9]{Lin18b}) and Puig's conjecture (\cite[Conjecture 6.4.2]{Lin18b}). Morita equivalences with endopermutation source were first systematically studied by Puig in \cite[Section 7]{Puig1999}. Like splendid Rickard equivalences and $p$-permutation equivalences, Morita equivalences with endopermutation source preserve the local structure of blocks. It was a folklore conjecture that any Morita equivalence between blocks of group algebras over $\O$ is induced by a bimodule with an endopermutation source. 
In order to describe the connection between Morita equivalences with endopermutation source and global isotypies, the author defined the notion of an {\it almost global isotypy} (see \cite[Definition 1.3]{H26:isotypy} or Definition \ref{defi: almost global isotypy} below) which is slightly weaker than a global isotypy but implies a cyclic isotypy. In \cite{HZ} and \cite[Theorem 1.7]{H26:isotypy}, we proved that Morita equivalences with endopermutation source imply cyclic isotypies and almost global isotypies. However, one may still wonder whether endopermutation-source Morita equivalences imply global isotypies in general. For this, in \cite[Proposition 1.5]{H26:isotypy}, we proved that if a bimodule of two block algebras of finite groups induces a Morita equivalence and has an endotrivial $\O \Delta Q_8$-source of $\O$-rank $3$, then this bimodule together with its slashed modules cannot give a global isotypy. However, as we remarked in \cite[Remark 1.6]{H26:isotypy}, a global isotypy is not necessarily coming from lifting slashed modules. On the other hand, we didn't give an example of two specific groups and blocks to support the situation of \cite[Proposition 1.5]{H26:isotypy}. So one of the motivations of this paper is to solve these two problems.

In \cite[Th\'eorème 5.2]{Broue1990}, Brou\'{e} showed that a nilpotent block with a defect group $P$ is cyclic isotypic to $\O P$. When using Brou\'{e}'s result, one may forget that this is only for cyclic isotypies. For example, in \cite[Remark 11.3]{Lin08}, Linckelmann remarked that ``A nilpotent block with defect group $P$ is known to be (globally) isotypic to the defect group algebra $\O P$". The second motivation of this paper is to show that this is not true in general. However, by our result \cite[Theorem 1.7]{H26:isotypy}, a nilpotent block with a defect group $P$ is always almost globally isotypic to $\O P$. It is well known that if $G$ is a $p$-nilpotent group, then any block of $\O G$ is nilpotent; see e.g. \cite[page 220]{Lin18b}. The first result of this paper is the following:

\begin{theorem}\label{theo:main}
Let $p=2$ and $G$ be the $2$-nilpotent group $\SU_3(2)\cong 3_{+}^{1+2} \rtimes Q_8$. There is a block algebra $B$ of $\O G$ satisfying both of the following:
\begin{enumerate}[{\rm (i)}]
	\item $B$ is Morita equivalent to $\O Q_8$ via a bimodule with an endotrivial source $V$ such that $\dim_k(k\otimes_\O V)=3$, and hence $k\otimes_\O V$ is not defined over the prime field $\F_2$.
	\item $B$ is not globally isotypic to $\O Q_8$. Hence they are not $2$-permutation equivalent and not splendidly Rickard equivalent. 
\end{enumerate}
\end{theorem}

Statements (i) and (ii) will be proved in Sections \ref{sec:Proof of (i)} and \ref{sec:proof of ii} respectively, after reviewing the structure of the group $\SU_3(2)$ in Section \ref{sec:Structure of SU_3(2)}. Theorem \ref{theo:main} separates Morita equivalences with endopermutation source from $p$-permutation equivalences and splendid Rickard equivalences, and also separates almost global isotypies from global isotypies. Thus, among the notions considered here, an almost global isotypy is the strongest one that follows in general from an endopermutation-source Morita equivalence. Using Theorem \ref{theo:main}, we can construct an even stronger example:

\begin{theorem}\label{coro:stronger example}
Let $p=2$. There exist finite groups $G$ and $G'$, a nilpotent block algebra $B$ of $\O G$ and a nilpotent block algebra $B'$ of $\O G'$ satisfying the following three conditions:
\begin{enumerate}[{\rm (i)}]
	\item  $B$ and $B'$ are isomorphic as $\O$-algebras.
	\item $B$ and $B'$ are Morita equivalent via a bimodule with an endotrivial source $V$ such that $\dim_k(k\otimes_\O V)=3$.
	\item $B$ and $B'$ are not globally isotypic.
\end{enumerate}
\end{theorem}

The proof of Theorem \ref{coro:stronger example} will be given at the end of Section \ref{sec:proof of ii}.

\medskip\noindent\textbf{Notation and terminology.} Let $G$ and $H$ be finite groups. For any $g_1,g_2\in G$, we denote by $[g_1,g_2]$ the element $g_1g_2g_1^{-1}g_2^{-1}\in G$.  We denote by $\Delta G$ the diagonal subgroup $\{(g,g)\mid g\in G\}$ of $G\times G$.  For $R\in \{\O,k\}$, whenever useful, we view an $RG$-module as an $R\Delta G$-module and vice versa via the canonical isomorphism $\Delta G\cong G$. An $\O G$-$\O H$-bimodule $M$ can be regarded as a left $\O (G\times H)$-module (and vice versa) via $(g,h)m=gmh^{-1}$, where $g\in G$, $h\in H$ and $m\in M$. If $M$ is indecomposable as an $\O G$-$\O H$-bimodule, then $M$ is indecomposable as an $\O(G\times H)$-module, hence has a vertex (in $G\times H$) and a source. By a {\it block} of $RG$ we mean a primitive central idempotent $b$ of $RG$, and $RGb$ is called a {\it block algebra} of $RG$. For the terminology vertex, source, defect group, endotrivial module, endopermutation module, relative trace map, source idempotent/algebra, Brauer pair, fusion system, and nilpotent block, we refer to the books \cite{Lin18a,Lin18b}. For the definition of a slashed module and related results, we refer to \cite{Biland,Bilandadv} or \cite{H24,H26:isotypy}.

\section{Structure of the group $\SU_3(2)$}\label{sec:Structure of SU_3(2)}

\begin{void}\label{void:two groups}	
{\rm Recall that the extraspecial group of order $27$ and exponent $3$ has presentation $$3_{+}^{\,1+2}=\langle x,y,z\mid x^3=y^3=z^3=1,\ [x,y]=z,\ [x,z]=[y,z]=1\rangle,$$
and the quaternion group of order $8$ has presentation
$$Q_8=\langle \alpha,\beta\mid \alpha^4=\beta^4=1,~~\alpha^2=\beta^2=[\alpha,\beta] \rangle.$$}
\end{void}

\begin{void}\label{void:SU_3(2)}
{\rm (i) We write the finite field $\F_4$ with $4$ elements as $\F_4=\{0,1,\omega,\omega^2\}$, where $\omega$ is a generator of the cyclic group $\F_4^\times$. Hence we have $\omega^2+\omega+1=0$ in $\F_4$. For an element $a\in \F_4$, we denote $a^2$ by $\bar{a}$. For a matrix $A=(a_{ij})_{1\leq i,j \leq 3}$ in $\GL_3(\F_4)$, we denote by  $\bar{A}$ the matrix $(\bar{a}_{ij})_{1\leq i,j\leq 3}$ and by $A^T$ the matrix $(a_{ji})_{1\leq i,j\leq 3}$. The identity element of $\GL_3(\F_4)$ is denoted by $I$. With this notation,
$$\SU_3(2):=\left\{A\in\SL_3(\F_4)\mid \bar{A}^{T}A=I\right\}.$$

(ii) Let
$$X=
\begin{pmatrix}
	1&0&0\\
	0&\omega&0\\
	0&0&\omega^2
\end{pmatrix},
~~~Y=
\begin{pmatrix}
	0&1&0\\
	0&0&1\\
	1&0&0
\end{pmatrix}~~~{\rm and}~~~
Z=\omega^2 I$$
be elements of $\GL_3(\F_4)$.
One checks that $X,Y,Z\in \SU_3(2)$, $X^3=Y^3=Z^3=I$, $[X,Y]=Z$ and $Z\in Z(\SU_3(2))$.
Hence $\langle X,Y\rangle$ is a subgroup of $\SU_3(2)$ which is isomorphic to $3_{+}^{\,1+2}$. Since $|\SU_3(2)|=216$, $\langle X,Y\rangle$ is a Sylow $3$-subgroup of $\SU_3(2)$. Moreover $\langle X,Y\rangle$ is normal in $\SU_3(2)$. (One can use group-theoretic arguments to show that $\SU_3(2)$ has a unique Sylow $3$-subgroup by noting that $|\SU_3(2)|=216$ and $|Z(\SU_3(2))|=3$; here we omit this.)

(iii) Let $$J=\begin{pmatrix}
	0&0&1\\
	0&1&0\\
	1&0&0
\end{pmatrix}~~~{\rm and}~~~C=
\begin{pmatrix}
	1&1&0\\
	1&1&1\\
	0&1&1
\end{pmatrix}$$ be elements of $\GL_3(\F_4)$. Then we have $\bar{C}^TC=J$. Noting that $\bar{C}^T=C$ and $C^4=J^2=I$, we have 
$$C\cdot\SU_3(2)\cdot C^{-1}=\left\{A\in\SL_3(\F_4)\mid \bar{A}^{T}JA=J\right\}.$$

(iv) For any $a,b\in \F_4$, we write $u(a,b):=
\begin{pmatrix}
	1&a&b\\
	0&1&a^2\\
	0&0&1
\end{pmatrix}$. A straightforward calculation shows that $u(a,b)\in C\cdot \SU_3(2) \cdot C^{-1}$ if and only if $b^2+b=a^3$. Let $P=\{u(a,b)\mid a^3=b^2+b\}$. Then $P$ is a subgroup of $C\cdot \SU_3(2)\cdot C^{-1}$. The equality $a^3=b^2+b$ implies that if $a=0$, then $b\in \{0,1\}$, and if $a\neq 0$, then $b\in \{\omega,\omega^2\}$.  Hence $|P|=8$ and $P$ is a Sylow $2$-subgroup of $C\cdot \SU_3(2)\cdot C^{-1}$. Let
$$
R_0=u(1,\omega^2)=
\begin{pmatrix}
	1&1&\omega^2\\
	0&1&1\\
	0&0&1
\end{pmatrix}~~~{\rm and}~~~
S_0=u(\omega^2,\omega)=
\begin{pmatrix}
	1&\omega^2&\omega\\
	0&1&\omega\\
	0&0&1
\end{pmatrix}.$$
One checks that 
$$R_0^2=S_0^2=
\begin{pmatrix}
	1&0&1\\
	0&1&0\\
	0&0&1
\end{pmatrix}=[R_0,S_0].$$
Hence $P=\langle R_0,S_0\rangle \cong Q_8$. 

(v) Let $G=C\cdot\SU_3(2)\cdot C^{-1}$. Let $X_0=CXC^{-1}$, $Y_0=CYC^{-1}$ and let 
$$N=\langle X_0,Y_0\rangle=C\cdot\langle X,Y\rangle \cdot C^{-1}.$$
According to (ii), $N$ is the unique normal Sylow $3$-subgroup of $G$ and $N\cong 3_{+}^{\,1+2}$. Since $|G|=216$, $P$ is a Sylow $2$-subgroup of $G$ and we have 
$$\SU_3(2)\cong G=NP\cong 3_{+}^{\,1+2}\rtimes Q_8.$$
}
\end{void}

\section{Proof of Theorem \ref{theo:main} (i)}\label{sec:Proof of (i)}
Throughout this section we assume that $p=2$. To prove Theorem \ref{theo:main}, it suffices to consider the group $C\cdot \SU_3(2)\cdot C^{-1}$ instead of $\SU_3(2)$. We divide the proof into several propositions. 

\begin{proposition}\label{prop:simple kG-module}
Let $G=C\cdot \SU_3(2)\cdot C^{-1}$ as defined in \ref{void:SU_3(2)} (v). Let $\bar{L}=k\otimes_{\F_4}\F_4^3$ be the natural $k\GL_3(\F_4)$-module afforded by the matrix realization.  Then $\bar{S}:=\Res_G^{\GL_3(\F_4)}(\bar{L})$ is a simple $kG$-module, and $\Res_N^G(\bar{S})$ is a simple $kN$-module.
\end{proposition}

\begin{proof}
It suffices to show that $\Res_N^G(\bar{S})=\Res_N^{\GL_3(\F_4)}(\bar{L})$ is simple. Since $N=C\cdot \langle X,Y\rangle \cdot C^{-1}$, it suffices to show that $\Res_{\langle X,Y\rangle}^{\GL_3(\F_4)}(\bar{L})$ is simple.  The three
eigenvalues $1,\omega,\omega^2$ of $X$ are distinct, so every
$X$-invariant subspace of $\bar{L}$ is spanned by a subset of the three coordinate
vectors.  The matrix $Y$ permutes these vectors transitively.  Consequently,
the only $\langle X,Y\rangle$-invariant subspaces are $0$ and $\bar{L}$. This completes the proof.
\end{proof}

Recall that for $\Lambda\in \{\O,k\}$ and $H$ a finite group, an {\it endotrivial $\Lambda H$-module} is a finitely generated $\Lambda$-free $\Lambda H$-module $W$ such that $W\otimes_\Lambda W^*\cong \Lambda\oplus W'$ for some projective $\Lambda H$-module $W'$.

\begin{proposition}\label{prop:restriction of simple kG-module}
Keep the notation of Proposition \ref{prop:simple kG-module}. Let $\bar{V}=\Res_P^G(\bar{S})$. Then $\bar{V}$ is an indecomposable endotrivial $kP$-module with the vertex $P$.
\end{proposition}

\begin{proof}
Let $t=R_0^2=S_0^2$.  From \ref{void:SU_3(2)} (iv) we have
$$t=
	\begin{pmatrix}
		1&0&1\\
		0&1&0\\
		0&0&1
	\end{pmatrix}~~~{\rm and}~~~Z(P)=\langle t\rangle.$$
Consider the regular left $kZ(P)$-module $kZ(P)$. Since $Z(P)$ is a cyclic group generated by $t$, the isomorphism class of a $kZ(P)$-module (say $M$) is completely determined by the similarity class of the $k$-linear endomorphism of $M$ induced by $t$. So using elementary linear algebra, one easily shows that
\begin{equation*}
	\Res_{Z(P)}^P(\bar{V})\cong k\oplus kZ(P).
\end{equation*}
By \cite[page 43, Lemma 4]{Alperin}, for any finite-dimensional
$kZ(P)$-module $\bar{U}$,
$kZ(P)\otimes_k \bar{U}$ is free as a $kZ(P)$-module. Hence we obtain
\[\begin{split}
	\Res_{Z(P)}^P\End_k(\bar{V})
	&\cong (k\oplus kZ(P))\otimes_k(k\oplus kZ(P))^*\\
	&\cong k\oplus (kZ(P))^{\oplus 4}.
\end{split}
\]
Consider the composition of the inclusion map $k\cdot\Id_{\bar{V}}\hookrightarrow\End_k(\bar{V})$ and the trace map $\Tr_{\bar{V}}:\End_k(\bar{V})\to k$ sending a linear endomorphism of $\bar{V}$ to its trace. Both maps are homomorphisms of $kP$-modules. Since $\dim_k(\bar{V})=3$ is invertible in $k$, we obtain that
$$\End_k(\bar{V})=k\cdot\Id_{\bar{V}}\oplus \ker(\Tr_{\bar{V}}),$$
as $kP$-modules. Write $\bar{U}=\ker(\Tr_{\bar{V}})$.
Using the Krull--Schmidt theorem in the preceding restriction implies
$$\Res_{Z(P)}^P(\bar{U})\cong(kZ(P))^{\oplus4}.$$
Note that the only nontrivial elementary abelian $2$-subgroup of $P\cong Q_8$ is
$Z(P)$.  Therefore, by Chouinard's criterion \cite[Theorem 5.2.4]{Benson}, $\bar{U}$ is projective as a $kP$-module.  Since $\dim_k(\bar{U})=8$ and $kP$ is indecomposable, it follows that $\bar{U}\cong kP$.  So we have proved
\begin{equation}\label{eq:end-w}
	\End_k(\bar{V})\cong k\oplus kP
\end{equation}
as $kP$-modules. So by definition, $\bar{V}$ is endotrivial. By (\ref{eq:end-w}), we have $\End_{kP}(\bar{V})\cong k\oplus (kP)^P$, where $(kP)^P=\{a\in kP\mid ua=a~{\rm for~any}~u\in P\}$. It is straightforward to check that
$$(kP)^P=(kP)_1^{P}:=\{a\in (kP)^P\mid a=\sum_{u\in P}ua'~{\rm for~some}~a'\in kP\}.$$ Then by \cite[Proposition 2.13.11 (iii)]{Lin18a}, the
endomorphism ring of $\bar{V}$ in the $k$-stable category of $kP$-modules is $k\cdot \Id_{\bar{V}}$.  As $\dim_k(\bar{V})<|P|$, the module $\bar{V}$ has no
nonzero projective summand, and hence it is indecomposable. Green's indecomposability theorem, together with $\dim_k(\bar{V})=3$, now shows that $P$ is the vertex of $\bar{V}$.
\end{proof}

\begin{proposition}\label{prop:source of simple kG-module}
Keep the notation of Propositions \ref{prop:simple kG-module}, \ref{prop:restriction of simple kG-module}. The $kP$-module $\bar{V}$ is a source of the simple $kG$-module $\bar{S}$.
\end{proposition}
\begin{proof} 
Let $\Tr_P^G$ be the relative trace map $$\End_{kP}(\bar{V})=\End_{kP}(\Res_P^G(\bar{S}))\to \End_{kG}(\bar{S}).$$
Since $[G:P]=|N|=27$ is invertible in $k$ and
$\Tr_P^G(\Id_{\bar{V}})=[G:P]\cdot\Id_{\bar{S}}=27\cdot\Id_{\bar{S}}$,
Higman's criterion \cite[Theorem 2.6.2 (ii),(v)]{Lin18a} shows that $\bar{S}$ is a direct summand of
$\Ind_P^G(\bar{V})$.  Since $\bar{V}$ is indecomposable with vertex $P$, it is a $kP$-source of $\bar{S}$.
\end{proof}

\begin{proposition}\label{prop:3-dim endotrivial}
Let $W$ be a $3$-dimensional endotrivial $kQ_8$-module. Then $W$ is not defined over $\F_2$. In particular, the module $\bar{V}$ in Proposition \ref{prop:restriction of simple kG-module} is not defined over $\F_2$.
	\end{proposition} 
	
\begin{proof}	
By \cite[Proposition 7.5.2 (i)]{Lin18b}, $\Res_{Z(Q_8)}^{Q_8}(W)$ is an endotrivial $kZ(Q_8)$-module. Since $Z(Q_8)$ is of order $2$, $k\oplus kZ(Q_8)$ is the unique $3$-dimensional endotrivial module, up to isomorphism. Hence $\Res_{Z(Q_8)}^{Q_8}(W)\cong k\oplus kZ(Q_8)$. So the unique involution $\alpha^2$ of $Q_8$ acts nontrivially on $W$. Since every nontrivial subgroup of $Q_8$ contains $\langle \alpha^2 \rangle=Z(Q_8)$, this implies that the representation $Q_8\to \GL(W)$ is faithful.
If $W\cong k\otimes_{\F_2}W_0$ for some $3$-dimensional $\F_2Q_8$-module $W_0$, then scalar extension preserves the kernel, so $W_0$ would give an embedding $Q_8\hookrightarrow\GL_3(\F_2)$.  This is impossible, because $|\GL_3(\F_2)|=168$, and its Sylow $2$-subgroups are conjugate to the upper unitriangular subgroup of $\GL_3(\F_2)$, which is dihedral of order $8$, not quaternion. 
\end{proof}

\begin{remark}
{\rm At this stage, since we have proved Proposition \ref{prop:source of simple kG-module}, we could finish the proof of Theorem \ref{theo:main} (i) by applying Puig's structure theorem of source algebras of nilpotent blocks (\cite[Theorem 8.11.5]{Lin18b}) together with \cite[Theorem 9.11.9]{Lin18b}. However, since the groups and blocks considered here are very specific, we aim to give a proof as elementary as possible. Therefore, we choose not to invoke these deep results.
}
\end{remark}

\begin{proposition}\label{prop:V is a source of S}
The simple $kG$-module $\bar{S}$ in Proposition \ref{prop:simple kG-module} lifts to an $\O$-free $\O G$-module $S$ with an endotrivial $\O P$-source $V:=\Res_P^G(S)$.
\end{proposition}

\begin{proof}
The group $G=N\rtimes P$ is $2$-solvable.  By the Fong--Swan lifting
theorem \cite{Fong1961,Swan1960} (see e.g. \cite[Theorem 10.1]{Navarro}), the simple module $\bar{S}$ lifts to an $\O$-free
$\O G$-module $S$ of rank $3$; thus $k\otimes_{\O}S\cong \bar{S}.$
Let $\rho:G\to\GL(S)$ be the corresponding representation and let
$V=\Res_P^G(S).$
We now lift the isomorphism \eqref{eq:end-w} from $k$ to $\O$.  Since $3\in\O^\times$, as in the proof of Proposition \ref{prop:restriction of simple kG-module}, we have
$$\End_{\O}(V)=\O\cdot\Id_V\oplus U,$$
as $\O P$-modules, where $U$ is the kernel of the trace map $\Tr_V:\End_{\O}(V)\to \O$.
The reduction modulo the maximal ideal of $\O$ gives $k\otimes_{\O}U\cong \bar{U}\cong kP$.  If $u_0\in U$
lifts a $kP$-generator of $\bar{U}$, then the $\O P$-homomorphism
$$\O P\to U$$
sending any $a\in \O P$ to $au_0\in U$ is an $\O P$-isomorphism after reduction. By the Nakayama lemma and the equality of
$\O$-ranks, it is already an isomorphism.  Consequently,
\begin{equation}\label{eq:end-v}
	\End_{\O}(V)\cong\O\oplus\O P
\end{equation}
as $\O P$-modules.  Thus by definition, $V$ is an endotrivial $\O P$-module of $\O$-rank $3$.
Since $k\otimes_\O V\cong \bar{V}$, $V$ is indecomposable with vertex $P$.  The same argument used in the proof of Proposition \ref{prop:source of simple kG-module} also shows that $V$ is an $\O P$-source of $S$.
\end{proof}

\begin{proposition}\label{prop:Morita equivalence}
Keep the notation of Proposition \ref{prop:simple kG-module}. Let $e$ be the block of $\O N$ corresponding to the simple $kN$-module $\Res_N^G(\bar{S})$. Then $e$ remains a block of $\O G$. Moreover, the block algebra $B=\O Ge$ is Morita equivalent to $\O P$ via a $B$-$\O P$-bimodule $M$.
\end{proposition}

\begin{proof}
The character of $\Res_N^G(S)$ is $G$-invariant because it is the restriction of the character
of $S$.  Hence $e$ is fixed by $G$ and is central in $\O G$. Let $\bar{e}$ be the image of $e$ in $kN$. Since $2\nmid |N|$, the algebra $kN$ is semisimple. As $k$ is a splitting field for $N$ and $\bar{e}$ is the primitive central idempotent associated with the simple $kN$-module ${\rm Res}_N^G(\bar{S})$, by Wedderburn's theorem, the structural homomorphism
$kN\bar{e}\to{\rm End}_k(\bar{S})$ is a $k$-algebra isomorphism.
Consider the structural homomorphism 
\begin{equation}\label{eq:a-end-v}
	\O Ne\to\End_{\O}(S)
\end{equation}
of $\Res_N^G(S)$. Reduction modulo the maximal ideal of $\O$ gives the structural homomorphism of $\Res_N^G(\bar{S})$. It follows from Nakayama's lemma that (\ref{eq:a-end-v}) is an $\O$-algebra isomorphism as well.

Set $B=\O Ge$.  Since $G=N\rtimes P$, we have
$B=\bigoplus_{u\in P}(\O Ne)u$.  The $\O$-linear map
\begin{equation}\label{eq:block-algebra-isomorphism}
\begin{split}
	B&\to\End_{\O}(S)\otimes_{\O}\O P,\\
	nue&\mapsto\rho(nu)\otimes u
	\qquad(n\in N,\ u\in P)
\end{split}
\end{equation}
is an $\O$-algebra isomorphism, where $\rho:G\to \GL(S)$ is the structural homomorphism of $S$.  Indeed, the multiplicativity follows from
$$(n_1u_1)(n_2u_2)=n_1({}^{u_1}n_2)u_1u_2\qquad(n_1,n_2\in N,\ u_1,u_2\in P)$$
and from the fact that $\rho$ is a group homomorphism, and the bijectivity
follows from the isomorphism \eqref{eq:a-end-v}.  In
particular,
$$B\cong\End_{\O}(S)\otimes_{\O}\O P\cong M_3(\O P).$$
Since $\O P$ has a unique block and $M_3(\O P)$ is Morita equivalent to $\O P$, the identity element of $M_3(\O P)$ is the unique nonzero central idempotent.  Hence $e$ is a block of $\O G$,
and $B$ is a block algebra.

Via the isomorphism (\ref{eq:block-algebra-isomorphism}), let
\begin{equation}\label{eq:Morita bimodule}
M=S\otimes_{\O}\O P
\end{equation}
be the standard $(B,\O P)$-bimodule with
$$(nu)(s\otimes a)=\rho(nu)s\otimes ua~~~{\rm and}~~~(s\otimes a)u=s\otimes au,$$
for any $n\in N$, $u\in P$, $s\in S$ and $a\in \O P$.
It is the usual column bimodule for $M_3(\O P)$ and therefore induces a
Morita equivalence between $B$ and $\O P$.
\end{proof}

\begin{proposition}\label{prop: M has V as source}
The $B$-$\O P$-bimodule $M$ in (\ref{eq:Morita bimodule}) has the endotrivial $\O \Delta P$-module $V=\Res_P^G(S)$ as a source. (Recall that we view any $\O P$-module as an $\O \Delta P$-module via the canonical isomorphism $\Delta P\cong P$.)
\end{proposition}

\begin{proof}
 Let
$\pi:G=N\rtimes P\to P$ be the canonical projection. For any $g\in G$, $u\in P$ and $m\in M$,  with
the convention $(g,u)m=gmu^{-1}$, the $\O(G\times P)$-action on $M$ is defined by 
\begin{equation*}
	(g,u)(s\otimes a)
	=\rho(g)s\otimes\pi(g)au^{-1},
\end{equation*}
for any $s\in S$ and $a\in \O P$. The projection $\Pi:M\to M$ onto $S\otimes 1$ along the
basis elements $S\otimes u$, $u\neq1$, is $\O \Delta P$-linear: $\Delta P$ acts on the
second tensor factor by conjugation.  A set of representatives for
$(G\times P)/\Delta P$ is $\{(g,1)\mid g\in G\}$.  If $g=nu$ with $n\in N$ and
$u\in P$, then $(g,1)\Pi(g,1)^{-1}$ is the projection of $M$ onto
$S\otimes u$.  Therefore
\begin{equation*}
	\Tr_{\Delta P}^{G\times P}(\Pi)
	=\sum_{g\in G}(g,1)\Pi(g,1)^{-1}
	=|N|\cdot\Id_M=27\cdot\Id_M.
\end{equation*}
Define $\eta:\Res_{\Delta P}^{G\times P}(M)\to \Res_{P}^G(S)=V$ with $\eta\left(\sum_{u\in P}s_u\otimes u\right)=s_1$,  and $\iota:V\to\Res_{\Delta P}^{G\times P}(M)$ with $\iota(s)=s\otimes1$. One checks directly that $\eta$ and $\iota$ are homomorphisms of $\O \Delta P$-modules and $\Pi=\iota\circ \eta$. Define
$\varphi:M\to \Ind_{\Delta P}^{G\times P}(V)$ with $\varphi(m)= \sum_{g\in G}(g,1)\otimes \eta\bigl((g,1)^{-1}m\bigr)$ for any $m\in M$, and define $\psi:\Ind_{\Delta P}^{G\times P}(V)\to M$ with $\psi((g, u)\otimes s)=(g\otimes u)\iota(s)$ for any $g\in G$, $u\in P$ and $s\in V$. Since $\iota$ is an $\O \Delta P$-homomorphism, $\psi$ is well defined. One checks that $\varphi$ and $\psi$ are homomorphisms of $\O(G\times P)$-modules and $\psi\circ\varphi=\Tr_{\Delta P}^{G\times P}(\Pi)=27\cdot\Id_M$. Since $27$ is invertible in $\O$, we have $\frac{1}{27}\psi\circ\varphi=\Id_M$. This implies that $\varphi$ is injective and
$$\Ind_{\Delta P}^{G\times P}(V)= \varphi(M)\oplus\ker(\psi)$$
as $\O(G\times P)$-modules. In particular, $M$ is isomorphic to a direct summand of $\Ind_{\Delta P}^{G\times P}(V)$.

 The bimodule $M$ is indecomposable, and $V$, viewed as an $\O \Delta P$-module, has vertex $\Delta P$.
Moreover, $V$ is a summand of $\Res_{\Delta P}^{G\times P}(M)$ through
$S\otimes1$.  It follows that $\Delta P$ is a vertex of $M$ and that $V$ is an
$\O \Delta P$-source of $M$.  
\end{proof}

Thus the Morita equivalence $M$ in Proposition \ref{prop:Morita equivalence} has the required endotrivial source $V$. Its reduction $\bar{V}$ is $3$-dimensional and hence, by Proposition \ref{prop:3-dim endotrivial}, is not defined over $\F_2$. Now we have completed the proof of Theorem \ref{theo:main} (i).

\section{Proof of Theorem \ref{theo:main} (ii)}\label{sec:proof of ii}

\begin{notation}
{\rm 

(i) Let $G$ and $H$ be finite groups, $b$ a block of $\O G$ and $c$ a block of $\O H$. Denote by $\Irr_K(G,b)$ (resp. $\IBr_K(G,b)$) the set of irreducible ordinary (resp. Brauer) characters of $G$ over $K$ associated with the block $b$. Denote by $\Z\Irr_K(G,b)$ the group of generalized ordinary characters of $G$ over $K$ associated with the block $b$, and denote by $\Z\IBr_K(G,b)$ the corresponding group of generalized Brauer characters. For simplicity, we refer to ordinary characters simply as characters. Following Brou\'{e}, a {\it perfect isometry} between $b$ and $c$ is a group isomorphism 
$\Phi:\Z\Irr_K(H,c)\cong \Z\Irr_K(G,b)$
satisfying certain conditions (see \cite[D\'{e}finition 1.4]{Broue1990} or \cite[Definition 9.2.2]{Lin18b}). By arithmetic properties of a perfect isometry, $\Phi$ induces an isomorphism
$\bar{\Phi}:\Z\IBr_K(H,c)\cong\Z\IBr_K(G,b)$
such that 
\begin{equation}\label{eq:ordinary-decomposition-compatibility}
	d_G\circ\Phi=\bar{\Phi}\circ d_H;
\end{equation}
see \cite[Corollary 9.2.7]{Lin18b}. Here $d_G$ and $d_H$ are the usual decomposition maps (see the beginning of \cite[Section 5.14]{Lin18a}). Such a perfect isometry $\Phi$ is an isometry with respect to the usual scalar products of ordinary characters.  The perfect isometry $\Phi$ determines a generalized character $\mu_\Phi$ of $G\times H$ over $K$ by the formula
\begin{equation}\label{equ:associated virtual character}
\mu_\Phi(g,h)
	=\sum_{\chi\in\Irr_K(H,c)}\Phi(\chi)(g)\chi(h)
	\end{equation}
	for all $(g,h)\in G\times H$.	
The defining conditions for a perfect isometry require the vanishing property
\begin{equation}\label{eq:perfect-character-vanishing}
	\mu_\Phi(g,h)=0
\end{equation}
whenever exactly one of $g$ and $h$ has order prime to $p$ but the other has order divisible by $p$; see \cite[Definition 9.2.1 (ii)]{Lin18b}. 

(ii) An $(\O Gb,\O Hc)$-bimodule $M$ induces, via the tensor functor $(K\otimes_\O M)\otimes_{KH}-$, a $\Z$-linear map $\Phi_M:\Z\Irr_K(H,c)\to \Z\Irr_K(G,b)$. If $M$ induces a Morita equivalence, then $\Phi_M$ is a perfect isometry; see \cite[Th\'eorème 3.1]{Broue1990} or \cite[Corollary 9.3.3]{Lin18b}.

(iii) Given a $p$-element $u$ of $G$ and a block $e$ of $kC_G(u)$, denote by $\hat{e}$ the unique block of $\O C_G(u)$ that lifts $e$. Denote by ${\rm Cl}_K(G)$ the set of $K$-valued class functions on $G$ and by ${\rm Cl}_K(C_G(u)_{p'})$ the set of $K$-valued class functions on $p$-regular elements of $C_G(u)$. For any $\chi\in {\rm Cl}_K(G)$ associated with the block $b$, define a class function $d_{(G,b)}^{(u,e)}(\chi)$ in ${\rm Cl}_K(C_G(u)_{p'})$ by setting $d_{(G,b)}^{(u,e)}(\chi)(s)=\chi(\hat{e}us)$ for all $p'$-elements $s$ in $C_G(u)$. 
}
\end{notation}

\begin{definition}[{\cite[Definition 1.3]{H26:isotypy}}]\label{defi: almost global isotypy}
{\rm Let $G$ and $H$ be finite groups, $b$ a block of $\O G$ and $c$ a block of $\O H$. An {\it almost global isotypy} between $b$ and $c$ consists of the following data:
\begin{enumerate}[{\rm (a)}]
	\item A maximal $\O Gb$-Brauer pair $(D,e)$ and a maximal $\O Hc$-Brauer pair $(E,f)$. For any subgroup $P$ of $D$, let $e_P$ be the unique block of $kC_G(P)$ such that $(P,e_P)\leq (D,e)$, and denote by $\hat{e}_P$ the unique block of $\O C_G(P)$ that lifts $e_P$. Similarly, for any subgroup $Q$ of $E$, let $f_Q$ be the unique block of $kC_H(Q)$ such that $(Q,f_Q)\leq (E,f)$, and denote by $\hat{f}_Q$ the unique block of $\O C_H(Q)$ that lifts $f_Q$. 
	\item An isomorphism $\varphi:D\cong E$ which is also an isomorphism between the fusion systems $\FF$ and $\G$ of $b$ and $c$ determined by the choice of the maximal Brauer pairs $(D,e)$ and $(E,f)$, respectively.
	\item A family of perfect isometries $\Phi_Q: \Z\Irr_K(C_H(Q),\hat{f}_Q)\cong\Z\Irr_K(C_G(P),\hat{e}_P)$, $Q\leq E$, where $P=\varphi^{-1}(Q)$.
	\item For any subgroup $Q$ of $E$ and any element $v\in C_E(Q)$, an integer $\varepsilon_{Q,v}\in\{\pm1\}$.
\end{enumerate}
These data are subject to the following conditions:
\begin{enumerate} [\rm (i)] 
	\item (Equivariance) For any isomorphism $\alpha:P_1\cong P_2$ in $\mathcal{F}$, letting $Q_1=\varphi(P_1)$, $Q_2=\varphi(P_2)$ and $\beta=\varphi\circ\alpha\circ\varphi^{-1}:Q_1\cong Q_2$ be the corresponding isomorphism in $\G$, we have $\hat{\alpha}\circ\Phi_{Q_1}=\Phi_{Q_2}\circ \hat{\beta}$. Here $\hat{\alpha}:\Z \Irr_K(C_G(P_1),\hat{e}_{P_1})\cong \Z \Irr_K(C_G(P_2),\hat{e}_{P_2})$ and $\hat{\beta}:\Z \Irr_K(C_H(Q_1),\hat{f}_{Q_1})\cong \Z \Irr_K(C_H(Q_2),\hat{f}_{Q_2})$ are the isomorphisms given by conjugation with elements $x\in G$ and $y\in H$ satisfying $\alpha(u)=xux^{-1}$ for all $u\in P_1$ and $\beta(v)=yvy^{-1}$ for all $v\in Q_1$. 
	\item (Compatibility) For any subgroup $Q$ of $E$ and any element $v\in C_E(Q)$, setting $S=Q\langle v\rangle$, $P=\varphi^{-1}(Q)$, $u=\varphi^{-1}(v)$ and $R=P\langle u\rangle$, we have  
	\begin{equation}\label{eq:compatibility equation}
	d_{(C_G(P),\hat{e}_P)}^{(u,e_R)}\circ \Phi_Q=\varepsilon_{Q,v}\cdot\bar{\Phi}_{S}\circ d_{(C_H(Q),\hat{f}_Q)}^{(v,f_S)}
	\end{equation}
	as maps from $\Z\Irr_K(C_H(Q),\hat{f}_Q)$ to $K\otimes_{\Z}\Z\IBr_K(C_G(R),\hat{e}_R)$.
\end{enumerate}
If $\varepsilon_{Q,v}=1$ for any subgroup $Q$ of $E$ and any element $v\in C_E(Q)$, then these data (a),(b),(c) are called a {\it global isotypy} between $b$ and $c$, which is exactly the notion in \cite[Definition 15.3]{BP}. If there exists an (almost) global isotypy between $b$ and $c$, then we say that $b$ and $c$ are (almost) globally isotypic.
}
\end{definition}

\begin{remark}\label{remark: after the definition of an almost isotypy}
	{\rm (i) The terminology ``global isotypy"  was first used by Boltje and Xu \cite[1.9 (b)]{BX}. \cite[Definition 15.3]{BP} is slightly different from Boltje and Xu's global isotypy, but they have many similarities. So it seems reasonable to call the notion in \cite[Definition 15.3]{BP} a global isotypy as well. 
		
(ii) It is easy to see that the composition of two (almost) global isotypies is an (almost) global isotypy, and the inverse of an (almost) global isotypy is an (almost) global isotypy. Hence being (almost) globally isotypic is an equivalence relation for blocks of finite group algebras. 
	}
\end{remark}

\begin{proposition}\label{prop:automorphism induces a global isotypy}
Let $G$ be a finite group, $b$ a block of $\O G$ and $\tau\in \Aut(G)$. We abusively use the same symbol $\tau$ to denote the $\O$-algebra automorphism of $\O G$ induced by $\tau$. Then $\tau$ induces a global isotypy between the blocks $b$ and $\tau(b)$ of $\O G$. In particular, if $G$ is a finite $p$-group, then any automorphism of $G$ induces a global self-isotypy of $\O G$.
\end{proposition}

\begin{proof}
This looks obvious, and one can also use \cite[Theorem 15.4]{BP} to obtain a proof. But for completeness, we give a direct proof.
Let $(D,e)$ be a maximal $\O Gb$-Brauer pair. Using definitions, one sees that the pair $(E,f):=(\tau(D),\tau(e))$ is a maximal $\O G\tau(b)$-Brauer pair, and $\varphi:=\tau|_D:D\to E$ is an isomorphism between the associated fusion systems. For every subgroup $Q$ of $E$, set $P=\tau^{-1}(Q)$. If $(P,e_P)\leq (D,e)$, then $(Q,\tau(e_P))\leq(E,f)$; hence $f_Q=\tau(e_P)$ and $\hat f_Q=\tau(\hat e_P)$.

For every $Q\leq E$, define a group isomorphism
$$\Phi_Q:\Z\Irr_K(C_G(Q),\hat f_Q)\cong
 \Z\Irr_K(C_G(P),\hat e_P)$$
by setting $\Phi_Q(\chi)(a)=\chi(\tau(a))$ for any $\chi\in\Z\Irr_K(C_G(Q),\hat f_Q)$ and any $a\in C_G(P)$. The restriction of $\tau$ is an $\O$-algebra isomorphism
$$\O C_G(P)\hat e_P\cong \O C_G(Q)\hat f_Q.$$
Restriction along this isomorphism induces $\Phi_Q$ and is a Morita equivalence. Thus $\Phi_Q$ is a perfect isometry by \cite[Th\'{e}or\`eme 3.1]{Broue1990}.

We verify the equivariance condition. Let $\alpha:P_1\cong P_2$ be an isomorphism in the fusion system of $b$, induced by conjugation with some $x\in G$. Let $Q_i=\tau(P_i)$. The corresponding isomorphism $\beta=\varphi\circ\alpha\circ\varphi^{-1}:Q_1\cong Q_2$ is induced by conjugation with $\tau(x)$. Let $c_x$ denote the group isomorphism $C_G(P_1)\cong C_G(P_2)$ induced by conjugation with $x$, and $c_{\tau(x)}$ the group isomorphism $C_G(Q_1)\cong C_G(Q_2)$ induced by conjugation with $\tau(x)$. Then we have
$$\tau\circ c_x=c_{\tau(x)}\circ\tau,$$
which in turn implies 
$$\hat\alpha\circ\Phi_{Q_1}=\Phi_{Q_2}\circ\hat\beta.$$

It remains to verify the compatibility condition. Let $Q\leq E$ and $v\in C_E(Q)$, and set
$$P=\tau^{-1}(Q),\qquad u=\tau^{-1}(v),\qquad S=Q\langle v\rangle,\qquad R=P\langle u\rangle.$$
Then $S=\tau(R)$ and $\hat f_S=\tau(\hat e_R)$. The isomorphism induced by $\Phi_S$ on generalized Brauer characters is
$$\bar{\Phi}_S:\Z\IBr_K(C_G(S),\hat f_S)\cong\Z\IBr_K(C_G(R),\hat e_R),$$
and it is given by
$$\bar{\Phi}_S(\theta)(a)=\theta(\tau(a))$$
for any $\theta\in\Z\IBr_K(C_G(S),\hat f_S)$ and any $a\in C_G(R)_{p'}$.
Consequently, for $\chi\in\Z\Irr_K(C_G(Q),\hat f_Q)$ and every $s\in C_G(R)_{p'}$, we have
\begin{align*}
 \bigl(d_{(C_G(P),\hat e_P)}^{(u,e_R)}\circ\Phi_Q\bigr)(\chi)(s)
 &=\Phi_Q(\chi)(\hat e_Rus)\\
 &=\chi\bigl(\tau(\hat e_Rus)\bigr)\\
 &=\chi(\hat f_Sv\tau(s))\\
 &=\bigl(\bar{\Phi}_S\circ
 d_{(C_G(Q),\hat f_Q)}^{(v,f_S)}\bigr)(\chi)(s).
\end{align*}
Thus the compatibility condition holds with $\varepsilon_{Q,v}=1$ for all $Q$ and $v$. Together with $(D,e)$, $(E,f)$ and $\varphi$, the family $(\Phi_Q)_{Q\leq E}$ is the global isotypy between $b$ and $\tau(b)$ induced by $\tau$. If $G$ is a $p$-group, then $\O G$ has a unique block, which is necessarily fixed by any automorphism of $G$, and the last statement follows.
\end{proof}

\begin{remark}\label{remark: the choice of data a and b}
{\rm Let $G$ and $H$ be finite groups, $b$ a block of $\O G$ and $c$ a block of $\O H$. 
	
(i) Any two choices of a maximal $\O Gb$-Brauer pair (respectively, a maximal $\O Hc$-Brauer pair) are conjugate by an element of $G$ (respectively, $H$), and such conjugation induces an inner automorphism of the corresponding group. By Proposition \ref{prop:automorphism induces a global isotypy} and Remark \ref{remark: after the definition of an almost isotypy} (ii), we see that the existence of an (almost) global isotypy between $b$ and $c$ does not depend on the choice of data (a) in Definition \ref{defi: almost global isotypy}. But it may depend on the choice of data (b).
	
(ii) Assume that one of $G$ or $H$ is a finite $p$-group. We may assume that $G$ is a finite $p$-group. Then there is a unique maximal $\O G$-Brauer pair $(G,1)$. Fix a choice of data (a): let $(E,f)$ be a maximal $\O Hc$-Brauer pair. Assume that there are two isomorphisms $\varphi_1:G\cong E$ and $\varphi_2: G\cong E$, both of which are isomorphisms between the fusion systems $\FF$ and $\G$ of $b$ and $c$ determined by the choice of the maximal Brauer pairs $(G,1)$ and $(E,f)$, respectively. Then $\varphi_1^{-1}\circ\varphi_2$ is an automorphism of $G$ which is also an automorphism of the fusion system $\FF$. By Proposition \ref{prop:automorphism induces a global isotypy}, $\varphi_1^{-1}\circ \varphi_2$ induces a global self-isotypy of $\O G$. Hence by Remark \ref{remark: after the definition of an almost isotypy} (ii), the existence of an (almost) global isotypy between $b$ and $c$ does not depend on the choice of data (b) in Definition \ref{defi: almost global isotypy}.
}
\end{remark}

\begin{proposition}\label{source algebra of B}
Keep the notation of Proposition \ref{prop:Morita equivalence}. The identity element $e$ of $B$ remains primitive in $B^P:=\{a\in B\mid uau^{-1}=a,~\forall u\in P\}$. Equivalently, $B$ is a source algebra of $B$.	
\end{proposition}

\begin{proof}
Using the explicit construction of the isomorphism (\ref{eq:block-algebra-isomorphism}), one checks that the isomorphism (\ref{eq:block-algebra-isomorphism}) is also an isomorphism of interior $P$-algebras. Since $\Res_P^G(S)$ is an indecomposable endotrivial $\O P$-module (see Proposition \ref{prop:V is a source of S}),  now the statement follows from \cite[Corollary 7.4.5]{Lin18b}.
\end{proof}

\begin{notation}\label{notation: fusion system in our case}
{\rm Keep the notation of Proposition \ref{prop:Morita equivalence}. Recall that $B=\O Ge$, $P\cong Q_8$, and the
$B$-$\O P$-bimodule $M$ induces a Morita equivalence between $B$ and $\O P$. Since $M$ has the $\Delta P$-source $V$ (see Proposition \ref{prop: M has V as source}), $M$ is isomorphic to a direct summand of 
$$B\otimes_{\O G}\Ind_{\Delta P}^{G\times P}(V)\otimes_{\O P}\O P\cong B\otimes_{\O P}\Ind_{\Delta P}^{P\times P}(V)\otimes_{\O P} \O P.$$
Let $(P,e_P)$ be the maximal $B$-Brauer pair associated to the source idempotent $e$ of $B$. By the proof of \cite[Theorem 9.11.2 (ii)]{Lin18b}, the fusion system on $P$ determined by the maximal $B$-Brauer pair $(P,e_P)$ is equal to the fusion system $\mathcal{F}_{P}(P)$. (See \cite[Example 8.1.2]{Lin18b} for meaning of the notation $\FF_P(P)$.) Alternatively, one can use the fact that $B$ is a nilpotent block to directly obtain this. The identity automorphism of $P$ defines an isomorphism between the fusion system on $P$ determined by the maximal $B$-Brauer pair $(P,e_P)$ and the fusion system on $P$ determined by the unique maximal $\O P$-Brauer pair $(P,1)$.  For every $Q\leq P$, let $e_Q$ be the unique block of
$kC_G(Q)$ determined by the source idempotent $e$ of $B$, and let $\hat e_Q$ be its lift to
$\O C_G(Q)$. If $Q\leq P$, $u\in C_P(Q)$, and $R=Q\langle u\rangle$, then we set
\begin{equation*}
	d_{G,Q}^{u}:=d_{(C_G(Q),\hat{e}_Q)}^{(u,e_R)}~~~{\rm and}~~~d_{P,Q}^{u}:=d_{(C_P(Q),1)}^{(u,1)}.
\end{equation*}
Thus, for any $\lambda\in {\rm Cl}_K(C_P(Q))$ associated with the block $1$ of $\O C_P(Q)$, $d_{P,Q}^{u}(\lambda)(s)=\lambda(us)$ for all $s\in C_P(R)_{2'}$. Since $C_P(R)$ is a $2$-group, $1$ is the unique $2'$-element of $C_P(R)$. In particular, if $\lambda$ is a linear character of $C_P(Q)$ and
$\varphi_R$ is the unique irreducible Brauer character of $C_P(R)$,
then
\begin{equation}\label{eq:decomposition-linear-character}
	d_{P,Q}^{u}(\lambda)=\lambda(u)\varphi_R.
\end{equation}

}
\end{notation}

We next record an elementary consequence of the property \eqref{eq:perfect-character-vanishing} that will allow us to treat an arbitrary candidate family of perfect isometries:

\begin{lemma}\label{lem:perfect-self-isometry-two-group}
Let $G$ be a finite $p$-group and let $\Phi$ be a perfect self-isometry of
$\Z\Irr_K(G,1)$.  There is a sign $\varepsilon\in\{\pm1\}$ and a
degree-preserving permutation $\sigma$ of $\Irr_K(G,1)$ such that
$\Phi(\chi)=\varepsilon\sigma(\chi)$ for any $\chi\in\Irr_K(G,1)$.
Denote by $1_G$ the trivial character of $G$ and by $\varphi_G$ the unique irreducible Brauer character of $G$. Then
$\bar{\Phi}(\varphi_G)=\varepsilon\varphi_G$ and
$\Phi(1_G)=\varepsilon\lambda$
for some linear character $\lambda$ of $G$.
\end{lemma}

\begin{proof}
Since $\Phi$ is an isometry of the $\Z$-module $\Z\Irr_K(G,1)$, it preserves elements in particular of norm $1$. Hence there is a permutation $\sigma$ of $\Irr_K(G,1)$ and signs
$\varepsilon_\chi\in\{\pm1\}$ such that
$\Phi(\chi)=\varepsilon_\chi\sigma(\chi)$.  Consider the $K$-valued class function
$$F_\Phi=\sum_{\chi\in\Irr_K(G)}\chi(1)\Phi(\chi)\in {\rm Cl}_K(G).$$
For $g\in G$ we have $F_\Phi(g)=\mu_\Phi(g,1)$; see (\ref{equ:associated virtual character}). Note that $1$ is the unique $p'$-element of $G$. Consequently, the property \eqref{eq:perfect-character-vanishing} says explicitly
that
$F_\Phi(g)=0$ for any $g\neq 1$.
The regular character $\rho_G$ of $G$ over $K$ has the same vanishing
property, and hence
$F_\Phi=c\rho_G$, where $c=\frac{F_\Phi(1)}{|G|}\in\mathbb{Q}$. (The regular character of $G$ over $K$ means the character of $G$ afforded by the left $KG$-module $KG$.)
Since the characters $\{\Phi(\chi)\mid \chi\in \Irr_K(G,1)\}$ form an orthogonal basis of $\Z\Irr_K(G,1)$,
$$\langle F_\Phi,F_\Phi\rangle_G
	=\sum_{\chi\in\Irr_K(G)}\chi(1)^2
	=|G|
	=\langle\rho_G,\rho_G\rangle_G.$$
It follows that $c^2=1$, so $c=\pm1$.  Comparing the coefficient of
$\sigma(\chi)$ in $F_\Phi=c\rho_G$ gives
$$\varepsilon_\chi\chi(1)=c\sigma(\chi)(1).$$
Both degrees are positive; therefore
$\varepsilon_\chi=c$ and $\sigma(\chi)(1)=\chi(1)$ for every $\chi\in \Irr_K(G,1)$.
Set $\varepsilon=c$. Now we have proved the first statement.

Finally, since $G$ is a $p$-group, $G$ has a unique irreducible Brauer character $\varphi_G$, and
the decomposition map $d_G$ satisfies $d_G(\chi)=\chi(1)\varphi_G$ for any $\chi\in \Irr_K(G,1)$.  By applying \eqref{eq:ordinary-decomposition-compatibility} to $\Phi$, we obtain
$$\bar{\Phi}(\chi(1)\varphi_G)
	=d_G(\Phi(\chi))
	=\varepsilon\sigma(\chi)(1)\varphi_G
	=\varepsilon\chi(1)\varphi_G.$$
Thus $\bar{\Phi}(\varphi_G)=\varepsilon\varphi_G$.  Since $\sigma$ preserves
degrees, $\Phi(1_G)$ is the common sign $\varepsilon$ times a linear character, proving the last statement.
\end{proof}

From now on we identify $P$ and $Q_8$ and write $$P=Q_8=\{1, z,\pm\alpha,\pm\beta,\pm\gamma\mid z^2=1, \alpha^2=\beta^2=\gamma^2=\alpha\beta\gamma=z\}.$$
From the character table of $Q_8$, we see that the values of any character of $Q_8$ are integers.  

\begin{lemma}\label{lem:q8-source-character}
Let $p=2$. Every linear character $\lambda$ of $P$ takes values in $\{\pm1\}$ and satisfies
\begin{equation}\label{eq:product-linear-character}
	\lambda(z)=1~~~{\rm and}~~~\lambda(\alpha)\lambda(\beta)\lambda(\gamma)=1.
\end{equation}
Let $\rho_V$ be the character of the source $V$ in Proposition \ref{prop: M has V as source} and $\theta$ the unique irreducible character of $P$ of degree $2$. Then
\begin{equation}\label{eq:character-v}
	\rho_V=\delta+\theta
\end{equation}
for some linear character $\delta$ of $P$. For any $u\in P$, let $\omega_V(u)$ be the sign of the integer $\rho_V(u)$. We have
\begin{equation}\label{eq:omega-v-values}
	\omega_V(u)=
	\begin{cases}
		-1, & u=z,\\
		\delta(u), & u\in\{\alpha,\beta,\gamma\}.
	\end{cases}
\end{equation}
\end{lemma}

\begin{proof}
Note that $[P,P]=\langle z\rangle$ and $P/[P,P]$ is a Klein four group. Hence every linear character is trivial
on $z$, has values in $\{\pm1\}$, and
$\lambda(\alpha)\lambda(\beta)\lambda(\gamma)=\lambda(\alpha\beta\gamma)=\lambda(z)=1$.

From the character table of $P$, we see that the four linear characters of $P$ all take the value $1$ at $z$.  The
remaining irreducible character $\theta$ has degree $2$ and satisfies
$$\theta(z)=-2~~~{\rm and}~~~\theta(\alpha)=\theta(\beta)=\theta(\gamma)=0.$$
Comparing the character values of the element $z$ on both sides of the isomorphism \eqref{eq:end-v}, we obtain $\rho_V(z)^2=1$.  Since $\rho_V(1)=3$, its decomposition into
irreducible characters is either the sum of three linear
characters or the sum of one linear character and $\theta$.  The first
possibility would give $\rho_V(z)=3$, contrary to
$\rho_V(z)^2=1$.  Thus \eqref{eq:character-v} holds. The equality \eqref{eq:omega-v-values} is a straightforward consequence of \eqref{eq:character-v}.
\end{proof}

\begin{notation}
{\rm Keep Notation \ref{notation: fusion system in our case}. We now recall some notation from \cite[Theorem 1.4]{H26:isotypy}.  For every nontrivial subgroup $Q$ of $P$, let $M_Q$ be a $(\Delta Q,e_Q\otimes 1)$-slashed module attached to $M$ over the group $C_G(Q)\times C_P(Q)$. Recall that the block algebras $B$ and $\O P$ have the fusion system $\FF_P(P)$, and any subgroup of $P$ is fully centralized in $\FF_P(P)$. Hence by \cite[Proposition 4.3]{H26:isotypy}, the $kC_G(Q)e_Q$-$kC_P(Q)$-bimodule $M_Q$ has $\Delta C_P(Q)$ as a vertex.  By \cite[Proposition 5.5]{H26:isotypy}, there is an $\O C_G(Q)\hat{e}_Q$-$\O C_P(Q)$-bimodule $\hat{M}_Q$ with an endopermutation $\Delta C_P(Q)$-source $\hat{V}_Q$ of determinant $1$, such that $k\otimes_\O\hat{M}_Q\cong M_Q$ and that $\hat{M}_Q$ induces a Morita equivalence between $\O C_G(Q)\hat{e}_Q$ and $\O C_P(Q)$. For $Q=1$, set $\hat M_1=M$ and $\hat V_1=V$. For any $Q\leq P$, denote by $\rho_{\hat{V}_Q}$ the character afforded by the $\O\Delta C_P(Q)$-module $\hat{V}_Q$, and for any $u\in C_P(Q)$, denote by $\omega_{\hat{V}_Q}(u)$ the sign of $\rho_{\hat{V}_Q}((u,u))$. Note that by \cite[Proposition 3.2]{H26:isotypy}, the values of $\rho_{\hat{V}_Q}$ are in $\Z$, so this makes sense. Now for any subgroup $Q$ of $P$, since the bimodule $\hat{M}_Q$ induces a Morita equivalence, it
induces a perfect isometry
\begin{equation*}
	\Phi_{\hat{M}_Q}:\Z\Irr_K(C_P(Q),1)\cong \Z\Irr_K(C_G(Q),\hat{e}_Q).
\end{equation*}
Recall that by \cite[Corollary 9.2.7]{Lin18b}, $\Phi_{\hat{M}_Q}$ induces an isomorphism
$$\bar{\Phi}_{\hat{M}_Q}:\Z\IBr_K(C_P(Q),1)\cong\Z\IBr_K(C_G(Q),\hat{e}_Q).$$
By \cite[Theorem 1.4 (ii) and equation (7.3)]{H26:isotypy}, the family $(\Phi_{\hat{M}_Q})_{\{Q\leq P\}}$ of perfect isometries forms an almost global isotypy and satisfies: for any $u\in C_P(Q)$ and $R=Q\langle u\rangle$,
\begin{equation}\label{eq:canonical-almost-compatibility}
	d_{G,Q}^{u}\circ\Phi_{\hat M_Q}
	=\omega_{\hat V_Q}(u)\bar{\Phi}_{\hat M_R}\circ d_{P,Q}^{u}.
\end{equation}

For later use, we spell out two special cases of
\eqref{eq:canonical-almost-compatibility}.  For $Q=1$ it gives,
\begin{equation}\label{eq:canonical-global-decomposition}
	d_{G,1}^{u}\circ\Phi_M
	=\omega_V(u)\bar{\Phi}_{\hat M_{\langle u\rangle}}
		\circ d_{P,1}^{u},
\end{equation}
for any $u\in P$. Set $Z=\langle z\rangle$.  By the
proof of \cite[Proposition 1.5]{H26:isotypy}, the $\O \Delta C_P(Z)$-module $\hat{V}_Z$ has $\O$-rank $1$. Since
$C_P(Z)=P$, the module $\hat{V}_Z$ affords a linear character of $\Delta C_P(Z)=\Delta P$, which can be viewed as a linear character of $P$ via the isomorphism $P\cong \Delta P$; we denote this linear character of $P$ by $\eta$.  Consequently, (\ref{eq:canonical-almost-compatibility}) gives
\begin{equation}\label{eq:canonical-z-decomposition}
	d_{G,Z}^{u}\circ\Phi_{\hat M_Z}
	=\eta(u)\bar{\Phi}_{\hat M_{\langle u\rangle}}
		\circ d_{P,Z}^{u}
	\qquad(u\in\{\alpha,\beta,\gamma\}).
\end{equation}
In particular, $\eta$ satisfies \eqref{eq:product-linear-character}.
}
\end{notation}

Now we start to prove Theorem \ref{theo:main} (ii):

\begin{proposition}\label{prop:no-strong-isotypy}
Keep Notation \ref{notation: fusion system in our case}. There is no global isotypy between $B$ and $\O P$ in the sense of
Definition~\ref{defi: almost global isotypy}.
\end{proposition}

\begin{proof}
Suppose that a global isotypy between $B$ and $\O P$ exists. By Remark \ref{remark: the choice of data a and b}, we may assume that data (a) and data (b) are those fixed in Notation \ref{notation: fusion system in our case}: the maximal $B$-Brauer pair is $(P,e_P)$, the maximal $\O P$-Brauer pair is $(P,1)$, and the isomorphism between defect groups is the identity automorphism of $P$. Denote the resulting family of perfect isometries by $(\Psi_Q)_{Q\leq P}$. For every $Q\leq P$, define
\begin{equation}\label{eq:normalising-self-isometry}
	T_Q=(\Phi_{\hat{M}_Q})^{-1}\circ\Psi_Q.
\end{equation}
Since perfect isometries are closed under inverse and composition, $T_Q:  \Z\Irr_K(C_P(Q),1)\to  \Z\Irr_K(C_P(Q),1)$
is a perfect self-isometry.  By
Lemma~\ref{lem:perfect-self-isometry-two-group}, there is a sign
$\varepsilon_Q\in\{\pm1\}$ such that $\bar{T}_Q(\varphi_Q)=\varepsilon_Q\varphi_Q$, where $\varphi_Q$ is the unique irreducible Brauer character of $C_P(Q)$.  Since
$C_P(1)=C_P(Z)=P$, the same lemma \ref{lem:perfect-self-isometry-two-group} gives linear characters
$\lambda_1$ and $\lambda_Z$ of $P$ such that
\begin{equation}\label{eq:t-one-and-z}
	T_1(1_P)=\varepsilon_1\lambda_1~~~{\rm and}~~~
	T_Z(1_P)=\varepsilon_Z\lambda_Z.
\end{equation}

$\bullet$ We first use the compatibility equation with $Q=1$ and $u\in C_P(Q)=C_P(1)=P$.  From
\eqref{eq:normalising-self-isometry},
$\Psi_Q=\Phi_{\hat M_Q}\circ T_Q$.  Hence the global isotypy equation
\eqref{eq:compatibility equation} reads
$$	d_{G,1}^{u}\circ\Phi_M\circ T_1
	=\bar{\Phi}_{\hat M_{\langle u\rangle}}
		\circ\bar{T}_{\langle u\rangle}\circ d_{P,1}^{u}.$$
Substituting \eqref{eq:canonical-global-decomposition} and cancelling
the isomorphism $\bar{\Phi}_{\hat M_{\langle u\rangle}}$ gives the equality of
maps
\begin{equation*}
	\omega_V(u)d_{P,1}^{u}\circ T_1
	=\bar{T}_{\langle u\rangle}\circ d_{P,1}^{u}.
\end{equation*}
Apply this equality to the trivial character $1_P$ of $P$.  By
\eqref{eq:decomposition-linear-character} and
\eqref{eq:t-one-and-z}, its left side is
$$\omega_V(u)\varepsilon_1\lambda_1(u)
	\varphi_{\langle u\rangle},$$
whereas its right side is
$\bar{T}_{\langle u\rangle}(\varphi_{\langle u\rangle})=\varepsilon_{\langle u\rangle}\varphi_{\langle u\rangle}$.  Therefore
\begin{equation}\label{eq:first-sign-relation}
	\varepsilon_{\langle u\rangle}
	=\varepsilon_1\omega_V(u)\lambda_1(u)
\end{equation}
for any $u\in P$. Taking $u=z$ in (\ref{eq:first-sign-relation}) and using Lemma~\ref{lem:q8-source-character} gives
\begin{equation}\label{eq:epsilon-z}
	\varepsilon_Z=-\varepsilon_1.
\end{equation}
 For $u\in\{\alpha,\beta,\gamma\}$, the equality (\ref{eq:first-sign-relation}) and Lemma~\ref{lem:q8-source-character} give
\begin{equation}\label{eq:epsilon-order-four}
	\varepsilon_{\langle u\rangle}
	=\varepsilon_1\delta(u)\lambda_1(u).
\end{equation}

$\bullet$ For $Q=Z$ and $u\in\{\alpha,\beta,\gamma\}$, we have
$Z\langle u\rangle=\langle u\rangle$.  Thus the global isotypy
compatibility condition \eqref{eq:compatibility equation} is
\begin{equation}\label{equ:compatibility for Z}
	d_{G,Z}^{u}\circ\Psi_Z =\bar{\Psi}_{\langle u\rangle}\circ d_{P,Z}^{u}.
	\end{equation}
By \eqref{eq:normalising-self-isometry}, we have
$$\Psi_Z=\Phi_{\hat M_Z}\circ T_Z~~~{\rm and}~~~\bar{\Psi}_{\langle u\rangle}=\bar{\Phi}_{\hat M_{\langle u\rangle}}\circ\bar{T}_{\langle u\rangle}.$$
Substituting these two expressions into (\ref{equ:compatibility for Z}), we obtain
$$	d_{G,Z}^{u}\circ\Phi_{\hat M_Z}\circ T_Z =\bar{\Phi}_{\hat M_{\langle u\rangle}}\circ\bar{T}_{\langle u\rangle}\circ d_{P,Z}^{u}.$$
Equation \eqref{eq:canonical-z-decomposition} then replaces
$d_{G,Z}^{u}\circ\Phi_{\hat M_Z}$ on the left side by
$\eta(u)\bar{\Phi}_{\hat M_{\langle u\rangle}}\circ d_{P,Z}^{u}$, giving
$$\eta(u)\bar{\Phi}_{\hat M_{\langle u\rangle}}
		\circ d_{P,Z}^{u}\circ T_Z
	=\bar{\Phi}_{\hat M_{\langle u\rangle}}
		\circ\bar{T}_{\langle u\rangle}\circ d_{P,Z}^{u}.$$
Cancelling the isomorphism $\bar{\Phi}_{\hat M_{\langle u\rangle}}$, we obtain
\begin{equation}\label{eq:normalised-z-compatibility}
	\eta(u)d_{P,Z}^{u}\circ T_Z
	=\bar{T}_{\langle u\rangle}\circ d_{P,Z}^{u}.
\end{equation}
We now evaluate both sides of
\eqref{eq:normalised-z-compatibility} at the trivial character $1_P$.
Recall from \eqref{eq:t-one-and-z} that $\lambda_Z$ is the linear character of $P=C_P(Z)$ determined by $T_Z(1_P)=\varepsilon_Z\lambda_Z$.
By \eqref{eq:decomposition-linear-character}, the left side is
$$\eta(u)d_{P,Z}^{u}\bigl(T_Z(1_P)\bigr)=\varepsilon_Z\eta(u)d_{P,Z}^{u}(\lambda_Z)=\varepsilon_Z\eta(u)\lambda_Z(u)\varphi_{\langle u\rangle}.$$
Here $\varphi_{\langle u\rangle}$ is the unique irreducible Brauer
character of $C_P(\langle u\rangle)$.
The right side is 
$$\bar{T}_{\langle u\rangle}\bigl(d_{P,Z}^{u}(1_P)\bigr)
	=\bar{T}_{\langle u\rangle}(\varphi_{\langle u\rangle})
	=\varepsilon_{\langle u\rangle}\varphi_{\langle u\rangle}.$$
Comparing the coefficients of $\varphi_{\langle u\rangle}$ gives
\begin{equation}\label{eq:second-sign-relation}
	\varepsilon_{\langle u\rangle}
	=\varepsilon_Z\eta(u)\lambda_Z(u).
\end{equation}
Combining \eqref{eq:epsilon-z}, \eqref{eq:epsilon-order-four} and
\eqref{eq:second-sign-relation}, and using $\eta(u)^{-1}=\eta(u)$, we
obtain
\begin{equation*}
	\lambda_Z(u)=-\delta(u)\lambda_1(u)\eta(u) \qquad(u\in\{\alpha,\beta,\gamma\}).
\end{equation*}
It follows that
$$\lambda_Z(\alpha)\lambda_Z(\beta)\lambda_Z(\gamma)=-\delta(\alpha)\delta(\beta)\delta(\gamma)\lambda_1(\alpha)\lambda_1(\beta)\lambda_1(\gamma)\eta(\alpha)\eta(\beta)\eta(\gamma).$$
Applying
\eqref{eq:product-linear-character} to $\lambda_Z$, $\delta$, $\lambda_1$ and $\eta$, the left side is $1$, while the right side is $-1$.  This contradiction proves
the proposition and hence Theorem \ref{theo:main} (ii).
\end{proof}

\begin{remark}
{\rm (i) In the proof of Proposition \ref{prop:no-strong-isotypy}, we only used the compatibility condition (ii) of Definition \ref{defi: almost global isotypy} - the equivariance condition (i) was not used. 
	
(ii) The particular choice $G\cong\SU_3(2)$ is not essential. Let $G$ be a finite group and let $b$ be a nilpotent block of $\O G$. Suppose that the unique simple $kG$-module belonging to the block $b$ has a $3$-dimensional endotrivial $kQ_8$-source. By \cite[Theorems 8.11.5, 9.11.9]{Lin18b}, there is an $\O Gb$-$\O Q_8$-bimodule $M$ inducing a Morita equivalence and having an endotrivial source $V$ of $\O$-rank $3$. The proof of Proposition \ref{prop:no-strong-isotypy} then applies verbatim and shows that \(\mathcal O Gb\) and \(\mathcal O Q_8\) are not globally isotypic. There are infinitely many examples of this form. For example, let $q$ be an odd prime and consider the extraspecial group of order $q^3$ and exponent $q$, which has presentation 
	$$q_{+}^{\,1+2}=\langle x,y,z\mid x^q=y^q=z^q=1,\ [x,y]=z,\ [x,z]=[y,z]=1\rangle.$$
There is an injective homomorphism $Q_8\hookrightarrow \SL_2(q)$ and an action of $\SL_2(q)$ on $q_+^{\,1+2}$; see the paragraph after \cite[Corollary 6.2]{Rognerud}. Let $G=q_+^{\,1+2}\rtimes Q_8$. If $q\equiv 3$ (mod $8$), then by \cite[Theorem 3.2 and Section 4.2]{Mazza}, there is a simple $kG$-module $\tilde{L}$ having a $3$-dimensional endotrivial source. One could also refer to the proof of \cite[Proposition 6.4]{Rognerud} for this. Since $G$ is a $2$-nilpotent group, any block of $\O G$ is nilpotent, see e.g. \cite[page 220]{Lin18b}. Then block $b$ of $\O G$ associated to $\tilde{L}$ is a desired block. This provides infinitely many such blocks. The choice $G\cong\SU_3(2)$ allows us to avoid invoking several deep results.
}
\end{remark}

\begin{proof}[Proof of Theorem \ref{coro:stronger example}]
Keep the notation of Propositions \ref{prop:simple kG-module}, \ref{prop:Morita equivalence} and \ref{prop: M has V as source}. Thus $G=N\rtimes P$, $B=\O Ge$, and the $B$-$\O P$-bimodule $M$ induces a Morita equivalence and has the endotrivial source $V=\Res_P^G(S)$, where $\dim_k(k\otimes_\O V)=3$.

Let $G'=N\times P$ and $B'=\O G'e$.
Since $\O P$ has a unique block, $e$ remains a block of $\O G'$, and hence $B'$ is a block algebra.
By \eqref{eq:a-end-v}, we have
$$B'\cong \O Ne\otimes_\O\O P\cong \End_\O(S)\otimes_\O\O P\cong M_3(\O P).$$
Comparing the above isomorphism with \eqref{eq:block-algebra-isomorphism}, we see that $B'$ and $B$ are isomorphic as $\O$-algebras. This proves (i).

Let $i'$ be a primitive idempotent of $\O Ne$. Since $\O Ne \cong \End_\O(S)\cong M_3(\O)$, we deduce that $i'\O N i'\cong \O$, which in turn implies 
$$i'B'i'\cong\O P$$ as interior $P$-algebras. This implies that $i'$ is a source idempotent of $B'$ and $i'B'i'$ is a source algebra of $B'$. On the other hand, by Proposition \ref{source algebra of B} and its proof, $B$ is a source algebra of itself and
$$B\cong\End_\O(V)\otimes_\O\O P$$
as interior $P$-algebras. It follows from \cite[Theorem 9.11.9]{Lin18b} that $B$ and $B'$ are Morita equivalent via a bimodule having $V$ as a source. This proves (ii).

It remains to prove (iii). Since $B'$ and $\O P$ have isomorphic source algebras as interior $P$-algebras, they are Morita equivalent via a bimodule with a trivial $\O\Delta P$-source; see e.g. \cite[Theorem 9.7.4]{Lin18b}, or even \cite[Theorem 9.11.2]{Lin18b}. Hence they are globally isotypic by \cite[Theorem 15.4]{BP}.

If $B$ and $B'$ were globally isotypic, then the transitivity of global isotypies, together with the preceding global isotypy between $B'$ and $\O P$, would imply that $B$ and $\O P$ are globally isotypic. This contradicts Proposition \ref{prop:no-strong-isotypy}. Therefore $B$ and $B'$ are not globally isotypic, proving (iii).
\end{proof}

\appendix
\renewcommand{\thetheorem}{\Alph{section}.\arabic{theorem}}
\renewcommand{\theequation}{\Alph{section}.\arabic{equation}}

\section{}\label{sec:counterexample}

 Since we used the group $3_{+}^{\,1+2}$ for Theorem \ref{theo:main}, we take this opportunity to give a \textbf{counterexample} to the following statement, which was stated in the literature (see \cite{PZ,Biland,Bilandadv,H24}) without proof:

\begin{statement}\label{statement:Brauer pair problem}
	Let $G$ be a finite group, $Q$ a $p$-subgroup of $G$ and $e$ a block of $kC_G(Q)$. For any subgroup $H$ of $G$ satisfying $C_G(Q)\leq H\leq N_G(Q,e)$, $e$ remains a block of $kH$.
\end{statement}

\begin{example}
{\rm Let $p=3$, and let $Q=3_{+}^{\,1+2}=\langle x,y,z\mid x^3=y^3=z^3=1,\ [x,y]=z,\ [x,z]=[y,z]=1\rangle$. 
Thus $Z(Q)=\langle z\rangle$.
Define an automorphism $\tau\in\operatorname{Aut}(Q)$ by
$$\tau(x)=x^{-1},~~~\tau(y)=y^{-1}~~~{\rm and}~~~\tau(z)=z.$$ 
Since $z\in Z(Q)$, we have
$[x^{-1},y^{-1}]=[x,y]=z$, so the defining relations are preserved and $\tau$ is indeed an automorphism of $Q$. Moreover, $\tau^2=1$.
Now we let $$G=Q\rtimes\langle a\rangle,$$ where $a$ is of order $2$ and conjugation by $a$ induces $\tau$ on $Q$.
We claim that
$$C_G(Q)=Z(Q)=\langle z\rangle.$$
Clearly $Z(Q)\le C_G(Q)$. On the other hand, if an element $ua$, with $u\in Q$, centralizes $Q$, then the inner automorphism induced by $ua$ on $Q$ would be trivial. Hence $\tau$ would be an inner automorphism of $Q$. However, every inner automorphism of $Q$ acts trivially on $Q/Z(Q)$,
whereas $\tau$ acts nontrivially on $Q/Z(Q)$. This is impossible. Therefore $C_G(Q)=Z(Q)$.
Since $C_G(Q)=Z(Q)$ is a $3$-group, the group algebra $kC_G(Q)$ has a unique block, namely $e:=1$.
Since \(Q\trianglelefteq G\) and \(e=1\), we have $N_G(Q,e)=G$. Now we set
$$H=\langle z,a\rangle.$$
Since $\tau(z)=z$, the elements $z$ and $a$ commute, and therefore
$H\cong\langle z\rangle\times\langle a\rangle$. In particular, $C_G(Q)=\langle z\rangle\le H\le N_G(Q,e)=G$.
However, $e=1$ is not a block of $kH$, because $kH\cong k\langle z\rangle\otimes_k k\langle a\rangle$ and $k\langle a\rangle$ has two blocks.}
\end{example}

Our counterexample shows that even if $Q$ is a Sylow $p$-subgroup of $G$, Statement \ref{statement:Brauer pair problem} could be \textbf{false}. Before discovering this counterexample, we consulted Robert Boltje about whether Statement \ref{statement:Brauer pair problem} is correct, and he told us that Philipp Perepelitsky had a counterexample to this statement long time ago which is unpublished. Since we don't know the detail of Perepelitsky's counterexample, we are not able to present it here.  We remark that if $QC_G(Q)\leq H\leq N_G(Q,e)$, the statement is true by \cite[Theorem 6.2.6 (iii)]{Lin18b}, and if $C_G(Q)\leq H\leq QC_G(Q)$, it is also true by \cite[Proposition 6.8.11]{Lin18b}. Even if Statement \ref{statement:Brauer pair problem} is not correct in general, it doesn't matter, because only the groups $C_G(Q)$, $QC_G(Q)$ and $N_G(Q,e)$ would be frequently used. 

\bigskip\noindent\textbf{Acknowledgements.} The author is grateful to Robert Boltje, Conghui Li, Markus Linckelmann, Lin Wu and Yuanyang Zhou for some helpful discussions. The author acknowledges support from NSFC (12501024, 12471016), China Postdoctoral Science Foundation (GZC20252006, 2025T001HB) and China Scholarship Council (202506770066).

\bigskip
{\footnotesize School of Mathematics and Statistics, Central China Normal University, Wuhan 430079, China 
	
	Email address: xinhuang@mails.ccnu.edu.cn}

\end{document}